\documentclass[reqno]{amsart}
\usepackage{amsmath,amssymb,amsfonts,amscd,mathrsfs}

\usepackage{fullpage}

\usepackage{enumitem}
\setlist[enumerate]{label = (\alph*)}

\numberwithin{equation}{section}

\theoremstyle{definition}
\newtheorem{theorem}{Theorem}[section]
\newtheorem{conjecture}[theorem]{Conjecture}
\newtheorem{corollary}[theorem]{Corollary} 
\newtheorem{definition}[theorem]{Definition} 
\newtheorem{example}[theorem]{Example}
\newtheorem{lemma}[theorem]{Lemma}
\newtheorem{proposition}[theorem]{Proposition}
\newtheorem{question}[theorem]{Question}

\DeclareMathOperator\ch{char}
\DeclareMathOperator\der{Der}
\DeclareMathOperator\gr{gr}
\DeclareMathOperator\GKdim{GKdim}
\DeclareMathOperator\gldim{gldim}
\DeclareMathOperator\Kdim{Kdim}

\DeclareMathOperator\Maxspec{MaxSpec}
\DeclareMathOperator\Mod{Mod}
\DeclareMathOperator\PMod{PMod}
\DeclareMathOperator\SW{Sw}
\DeclareMathOperator\Tdeg{Tdeg}
\DeclareMathOperator\Trdeg{Trdeg}

\newcommand\derh{\mathrm{Der}^\mathrm{H}}
\newcommand\pder{\mathrm{PDer}}

\newcommand\lnd{\mathrm{LND}}
\newcommand\plnd{\mathrm{PLND}}
\newcommand\plndh{\plnd^\mathrm{H}}

\newcommand\pml{\mathrm{PML}}

\newcommand\pcnt{\mathcal{Z}_P}
\newcommand\inv{^{-1}}
\newcommand\niso{\ncong}
\newcommand\iso{\cong}
\newcommand\tensor{\otimes}
\newcommand\kk{{\Bbbk}}

\newcommand\cP{\mathcal P}

\newcommand\DD{\mathbb D}
\newcommand\FF{\mathbb F}
\newcommand\NN{\mathbb N}
\newcommand\QQ{\mathbb Q}
\newcommand\ZZ{\mathbb Z}

\newcommand\fg{\mathfrak g}
\newcommand\fm{\mathfrak m}
\newcommand\fp{\mathfrak p}
\newcommand\fsl{\mathfrak{sl}}

\newcommand\restr[2]{{
  \left.\kern-\nulldelimiterspace 
  #1 
  \vphantom{\big|} 
  \right|_{#2} 
}}

\begin{document}

\title{Cancellation and skew cancellation for Poisson algebras}

\author[Gaddis]{Jason Gaddis}
\address{Miami University, Department of Mathematics, 301 S. Patterson Ave., Oxford, Ohio 45056} 
\email{gaddisj@miamioh.edu}

\author[Wang]{Xingting Wang}
\address{Howard University, Department of Mathematics, 204 Academic Service Building B, Washington DC, 20059} 
\email{xingting.wang@howard.edu}

\author[Yee]{Daniel Yee}
\address{Bradley University, Department of Mathematics, 1501 W. Bradley Ave., Peoria, IL 61625}
\email{dyee@bradley.edu}

\date{\today}
\begin{abstract}
We study the Zariski cancellation problem for Poisson algebras in three variables. In particular, we prove those with Poisson bracket either being quadratic or derived from a Lie algebra are cancellative. We also use various Poisson algebra invariants, including the Poisson Makar-Limanov invariant, the divisor Poisson subalgebra, and the Poisson stratiform length, to study the skew cancellation problem for Poisson algebras.  
\end{abstract}

\maketitle

\section{Introduction}
The original ``Zariski cancellation problem" was asked by Zariski in 1949. Geometrically, the question is whether an isomorphism $V\times \mathbb A^1\cong W\times \mathbb A^1$, with two $\kk$-affine varieties $V$ and $W$, $\kk$ a field, implies that $V \iso W$ as $\kk$-affine varieties. A noncommutative version of the Zariski cancellation problem was investigated in as early as 1970s by Coleman-Enochs \cite{CE71} and Brewer-Rutter \cite{BR72}, and was re-introduced by Bell and Zhang in 2017 \cite{BZ2}. Recently, the Zariski cancellation problem has attracted significant interest in the noncommutative setting of coordinate rings of noncommutative algebraic varieties in noncommutative projective algebraic geometry; see \cite{BHHV,BZ2,Bcancel,LWZ1,LWZ2,TZZ}. 

The notion of the Poisson bracket was introduced by French mathematician Sim\'eon Denis Poisson in the search for integrals of motion in Hamiltonian mechanics \cite{CAP2013}. Nowadays, it is deeply entangled with noncommutative geometry, integrable systems, topological field theories, and representation theory of noncommutative algebras.  Very recently, in the spirit of deformation quantization, the cancellative properties for Poisson algebras were introduced by the first and second authors in \cite{GXW}. By applying the tools of nilpotent derivations introduced by Makar-Limanov and noncommutative discriminant introduced by Zhang and et. al. in the setting of Poisson algebras, various families of Poisson algebras are shown to be cancellative therein. The Makar-Limanov invariant was originally introduced by Makar-Limanov \cite{Ma96} and has become a very useful invariant in commutative algebra. The discriminant method was introduced in \cite{CPWZ15, CPWZ16} to calculate the automorphism groups for a class of noncommutative algebras. But several questions are also raised regarding whether or not the cancellation properties hold for other well-known Poisson algebras including the Poisson polynomial algebra in three variables with nontrivial Jacobian bracket \cite[Question 8.4]{GXW}.

This present paper continues the study of the Zariski cancellation problem for Poisson algebras in various new directions. 

In Section \ref{section3}, we investigate the cancellation properties of Poisson polynomial algebras in three variables. First of all, in Theorem \ref{thm.filtered} we classify all filtered quadratic Poisson algebras with two variables, which yields those quadratic Poisson algebras with three variable as their homogenizations in Theorem \ref{thm.homog}. Our classification extends the work of Dumas \cite{dumas} on the classification of filtered quadratic Poisson algebras up to rational equivalence. Also the classifications align with that of filtered Artin-Schelter regular algebras of dimension two and that of their homogenizations \cite{Gtwogen}. Next, by analyzing the Poisson center of an arbitrary connected graded Poisson domain (Lemma \ref{lem.notkt} and Lemma \ref{lem.kt}), we are able to prove Poisson polynomial algebras in three variables with Poisson bracket either being quadratic or derived from a Lie algebra are cancellative. 

\begin{theorem}
Let $\kk$ be a base field that is algebraically closed of characteristic zero. 
\begin{enumerate}
    \item (Theorem \ref{thm.3var} and Corollary \ref{cor:veronese}) Let $A=\kk[x,y,z]$ be a quadratic Poisson algebra with nontrivial Poisson bracket. Then the $d$th Veronese Poisson subalgebra $A^{(d)}$ is Poisson cancellative for any $d\ge 1$. In particular, $A$ is Poisson cancellative. 
    \item (Theorem \ref{thm:Lie}) Let $\mathfrak g$ be a non-abelian Lie algebra of dimension $\le 3$. Then the symmetric algebra $S(\mathfrak g)$ on $\mathfrak g$ together with the Konstant-Kirillov bracket is Poisson cancellative. 
\end{enumerate}
\end{theorem}

In Section \ref{section4}, we introduce the \emph{skew Zariski cancellation problem} for Poisson algebras in terms of any Poisson-Ore extension $A[t;\alpha,\delta]_P$ (see subsection \ref{sec:POE} for definition). That is, it asks whether or not an isomorphism between the Poisson-Ore extensions of two base Poisson algebras implies an isomorphism between the two base Poisson algebras. We explore various invariants, including the Poisson Makar-Limanov invariant, the divisor Poisson subalgebra, and the Poisson stratiform length, which play important roles in the skew cancellation problem.  Relevant papers in the noncommutative algebraic setting are  \cite{BHHV,Bcancel,TZZ}.

\begin{theorem}
Let $\kk$ be a base field of characteristic zero, and $A$ be a noetherian Poisson domain.
\begin{enumerate}
    \item (Theorem \ref{thm.simple}) If $A$ is either Poisson simple of finite Krull dimension with units in  $\kk$ or affine of Krull dimension one, then $A$ is Poisson $\alpha$-cancellative.
\item (Theorem \ref{thm.delta})  If the Poisson Makar-Limanov invariant of $A$ equals $A$, then $A$ is Poisson $\delta$-cancellative. 
\item (Theorem \ref{thm.dd1})  If $A$ has finite Krull dimension and the $1$-divisor Poisson subalgebra  of $A$ equals $A$, then $A$ is Poisson skew cancellative.
\item (Theorem \ref{thm:stratiform})
If $A$ is Poisson stratiform and the $1$-divisor Poisson subalgebra of $A$ equals $A$, then $A$ is Poisson skew cancellative in the category of noetherian Poisson stratiform domains.
\end{enumerate}
\end{theorem}

Basic notions and properties related to Poisson algebras are reviewed in Section \ref{section2}. In particular, we introduce different cancellation properties for Poisson algebras in subsection \ref{def:Poissoncancellation} and many useful invariants for Poisson algebras including the Poisson discriminant and the Poisson Makar-Limanov invariant in subsection \ref{PD} that lead us to prove various families of Poisson algebras are cancellative in this paper.

\section{Background}\label{section2}

Throughout the paper, we work over a base field $\kk$. All algebras are $\kk$-algebras and all tensor products are over $\kk$.

\subsection{Poisson algebras}

A \emph{Poisson algebra} is a commutative $\kk$-algebra $A$
equipped with a bilinear map $\{-,-\}: A \times A \rightarrow A$, called a \emph{Poisson bracket}, which is both a Lie bracket and a biderivation. 

Suppose $(A,\{-,-\})$ is a Poisson algebra. A subalgebra $B$ of $A$ is a \emph{Poisson subalgebra}
of $A$ if $\{b,b'\} \subseteq B$ for all $b,b' \in B$. A \emph{homomorphism} $\phi:A \rightarrow C$ of Poisson algebras is an algebra homomorphism such that for all $a_1,a_2 \in A$, $\phi(\{a_1,a_2\}_A) = \{\phi(a_1),\phi(a_2)\}_C$.
The \emph{Poisson center} of $A$ is denoted by 
\[\pcnt(A) := \{ z \in A : \{z,a\}=0 \text{ for all } a \in A \}.\]
We say $A$ has \emph{trivial Poisson center} if $\pcnt(A)=\kk$.

A \emph{Poisson ideal} $I$ of a Poisson algebra $A$ is an ideal such that $\{I,A\} \subseteq I$. In this setting, there is a natural Poisson bracket on $A/I$ given by 
\[ \{a+I,b+I\}:=\{a,b\}+I\]
for all $a+I,b+I \in A/I$. Moreover, the natural quotient map $\pi:A \to A/I$ is a surjective Poisson algebra homomorphism. Any Poisson ideal $P$ is called \emph{Poisson prime} if for all Poisson ideals $I$ and $J$ with $IJ \subseteq P$, $I \subset P$ or $J \subseteq P$.

The Poisson bracket on $A$ extends naturally to the localization $AS^{-1}$ for any multiplicative set $S$ of $A$ by 
\[
\left\{a/s,b/t\right\}:=(1/st)\{a,b\}-(a/s^2t)\{s,b\}-(b/t^2s)\{a,t\}+(ab/s^2t^2)\{s,t\}
\]
for any $a,b\in A$ and $s,t\in S$. If $S=\{1,t,t^2,\ldots\}$ for some $t\in A$, we will simply write $A[t^{-1}]=AS^{-1}$.

Given Poisson algebras $A$ and $B$ with brackets $\{,\}_A$ and $\{,\}_B$, respectively, the tensor product $A \tensor B$ is naturally equipped with a Poisson bracket
\[ 
\{a_1 \tensor b_1, a_2 \tensor b_2\}
    = a_1a_2 \tensor \{b_1,b_2\}_B + \{a_1,a_2\}_A \tensor b_1b_2
\]
for all $a_1,a_2 \in A$ and $b_1,b_2 \in B$. In particular, this construction may be applied to give a natural Poisson structure on a polynomial extension $A[x_1,\hdots,x_n] = A \tensor \kk[x_1,\hdots,x_n]$ in which we always assume that $\kk[x_1,\hdots,x_n]$ has the zero Poisson bracket.

\subsection{Poisson Derivations and Poisson-Ore extensions}\label{sec:POE}
A \emph{derivation} of an algebra $A$ is a $\kk$-linear endomorphism $\alpha$ of $A$ such that $\alpha(ab)=\alpha(a)b+a\alpha(b)$ for all $a,b\in A$. The derivation $\alpha$ is \emph{locally nilpotent} if for every $a \in A$ we have $\alpha^n(a)=0$ for $n\gg 0$. We denote the collection of derivations (resp. locally nilpotent derivations) of $A$ by $\der(A)$ (resp. $\lnd(A)$). 

In positive characteristic, it is usually more convenient to replace  derivations with higher derivations. A higher derivation (or Hasse-Schmidt derivation) on $A$ is a sequence of $\kk$-linear endomorphisms $\partial :=\{\partial_i\}_{i=0}^\infty$ such that
\[
\partial_0={\rm id}_A,\quad \partial_n(ab)=\sum_{i=0}^n \partial_i(a)\partial_{n-i}(b)\quad \text{for all $a,b\in A$ and all $n \ge 1$}.
\]
The collection of higher derivations of $A$ is denoted by $\derh(A)$.

Now suppose $A$ is a Poisson algebra. A derivation $\alpha$ of $A$ is called a \emph{Poisson derivation} if additionally we have 
\[ \alpha(\{a,b\}) = \{\alpha(a),b\} + \{a,\alpha(b)\}\]
for all $a,b \in A$. We denote the collection of Poisson derivations (resp. locally nilpotent Poisson derivations) of $A$ by $\pder(A)$ (resp. $\plnd(A)$).

In positive characteristic, there is a notion of higher Poisson derivations introduced by Launois and Lecoutre [23], which is a higher derivation $\{\partial_i\}_{i=0}^\infty$ on $A$ satisfying
\[
\partial_n(\{a,b\})=\sum_{i=0}^n \{\partial_i(a),\partial_{n-i}(b)\}\quad \text{for all $a,b\in A$ and all $n \ge 0$}.
\]
We denote the collection of locally nilpotent higher Poisson derivations of A by $\plndh(A)$. 

Given a Poisson derivation $\alpha$ of $A$, a derivation $\delta$ of $A$ is a 
\emph{Poisson $\alpha$-derivation} if it satisfies
\[ \delta(\{a,b\})=\{\delta(a),b\}+\{a,\delta(b)\}+\alpha(a)\delta(b)-\delta(a)\alpha(b)\]
for all $a,b\in A$. Note that a Poisson $\alpha$-derivation is just a Poisson derivation when $\alpha=0$. 

Given a Poisson derivation $\alpha$ and a Poisson $\alpha$-derivation $\delta$ of $A$, the \emph{Poisson-Ore extension} $A[t;\alpha,\delta]_P$ is the polynomial ring $A[t]$ together with the following Poisson bracket
\[  \{a,b\} = \{a,b\}_A, \qquad \{a,t\} = \alpha(a)t +\delta(a),\]
for all $a,b\in A$. When $\delta = 0$ (resp. $\alpha=0$), we abbreviate $A[t;\alpha,\delta]_P$ to $A[t;\alpha]_P$ (resp. $A[t;\delta]_P$).
The Poisson structure of $A[t;\alpha,\delta]_P$ on $A[t]$ extends (uniquely) to a Poisson structure on $A(t)$, the algebra of rational functions with coefficients in $A$, and we denote this Poisson algebra by $A(t;\alpha,\delta)_P$.

\begin{lemma}
Let $A$ be a Poisson algebra and let $A[t_1,\hdots,t_n]$ be a Poisson algebra over $A$.
If $\{A,t_1\}\subseteq At_1+A$ and 
\[
\{A[t_1,\hdots,t_i],t_{i+1}\}\subseteq A[t_1,\hdots,t_i]\,t_{i+1}+A[t_1,\hdots,t_i]
\]
for all $i=1,\hdots,n-1$, then 
\[A[t_1,\hdots,t_n]\cong A[t_1;\alpha_1,\delta_1]_P\cdots [t_n;\alpha_n,\delta_n]_P\]
is an iterated Poisson-Ore extensions over $A$.
\end{lemma}
\begin{proof}
This follows in the case $n=1$ by \cite[Theorem 1]{OH1}. The stated result follows easily by induction.
\end{proof}

\subsection{Poisson cancellation}\label{def:Poissoncancellation}
We review definitions related to Poisson cancellation. Throughout, $A$ is a Poisson algebra with Poisson center $\pcnt=\pcnt(A)$. The word \emph{strongly} in the following definitions can be dropped by restricting to the case $n=1$.
\begin{enumerate}
    \item \label{univ} We say $A$ is \emph{universally Poisson cancellative} if $A \tensor R \iso B \tensor R$ implies $A \iso B$ for every Poisson algebra $B$ and every finitely generated commutative domain $R$ over $\kk$ with trivial Poisson bracket such that the natural map $\kk \rightarrow R \rightarrow R/I$ is an isomorphism for some Poisson ideal $I \subseteq R$.
    
    \item In the special case of \ref{univ} when $R=\kk[t_1,\hdots,t_n]$ for all $n \geq 1$, we say $A$ is \emph{strongly Poisson cancellative}.
    
    \item We say $A$ is \emph{strongly Poisson detectable} (resp. \emph{strongly $\pcnt$-detectable}) if, for any Poisson algebra $B$ and integer $n \geq 1$, any Poisson algebra isomorphism $\phi:A[t_1,\hdots,t_n] \rightarrow B[s_1,\hdots,s_n]$ implies that $B[s_1,\hdots,s_n] = B[\phi(t_1),\hdots,\phi(t_n)]$, or equivalently, $s_1,\hdots,s_n \in B[\phi(t_1),\hdots,\phi(t_n)]$ (resp. $\pcnt(A)[s_1,\hdots,s_n] = \pcnt(B)[\phi(t_1),\hdots,\phi(t_n)]$).
    
    \item We say that $A$ is \emph{strongly Poisson retractable} (resp. \emph{strongly $\pcnt$-retractable}) if, for any Poisson algebra $B$ and integer $n \geq 1$, any Poisson algebra isomorphism $\phi:A[t_1,\hdots,t_n] \rightarrow B[s_1,\hdots,s_n]$ implies that $\phi(A)=B$ (resp. $\phi(\pcnt(A)) = \pcnt(B)$).
    
    \item \label{skcancel} We say $A$ is \emph{strongly Poisson skew cancellative} if any Poisson isomorphism 
\[ A[t_1;\alpha_1,\delta_1]_P \cdots [t_n;\alpha_n,\delta_n]_P \iso A'[t_1';\alpha_1',\delta_1']_P \cdots [t_n';\alpha_n',\delta_n']_P\] 
    between two iterated Poisson-Ore extensions of Poisson algebras $A$ and $A'$ of the same length $n\ge 1$ implies a Poisson isomorphism $A \iso A'$. 
    
    \item In the special case of \ref{skcancel} in which $\delta_i=\delta_i'=0$ for all $i$, we say $A$ is \emph{strongly Poisson $\alpha$-cancellative}.
    
    \item In the special case of \ref{skcancel} in which $\alpha_i=\alpha_i'=0$ for all $i$, we say $A$ is \emph{strongly Poisson $\delta$-cancellative}.
\end{enumerate}

Note that in all cases above, our definitions can be viewed as Poisson versions of the corresponding definitions for cancellation problems for associative algebras. 

\subsection{Poisson discriminant and rigidity}
\label{PD}
For Poisson algebras, the notion of Poisson discriminant has been introduced in \cite{GXW} in order to tackle the Poisson cancellation problem via the Poisson Makar-Limanov invariant and the corresponding rigidity. 

Let $A$ be a Poisson algebra. By a Poisson property $\cP$ we mean a property that is invariant under Poisson isomorphism within a class of Poisson algebras. 
The {$\cP$-locus} of $A$ is 
\[
L_{\cP}(A):=\{\fm\in \Maxspec(\pcnt(A) ): A/\fm A \text{ has property $\cP$}\}
\]
and the \emph{$\cP$-discriminant set} of $A$ is
\[
D_{\cP}(A):=\Maxspec(\pcnt(A))\setminus L_{\cP}(A).
\]
Now we can define the \emph{$\cP$-discriminant ideal} of $A$ to be
\[
I_{\cP}(A):=\bigcap_{\fm\in D_{\cP}(A)} \fm\subset \pcnt(A).
\]

In the special case such that $I_{\cP}(A)$ is a principal ideal of $\pcnt(A)$, say generated by some element $d\in \pcnt(A)$, then we call $d$ the \emph{$\cP$-discriminant} of $A$, denoted by $d_{\cP}(A)$. When $\pcnt(A)$ is a domain, $d_{\cP}(A)$ is unique up to some unit element of $\pcnt(A)$. 

For a higher Poisson derivation $\partial=(\partial_i)_{i=0}^\infty$, the \emph{kernel} of $\partial$ is defined to be 
\[
\ker(\partial)=\bigcap_{i\ge 1} \ker(\partial_i).
\]
Let $0\neq d\in \pcnt(A)$, and let $*$ be either blank or $H$. The set of all \emph{(higher) locally nilpotent Poisson derivations of $A$ with respect to $d$} is denoted by
\[
{\plnd^*_d}(A):=\{\partial\in\plnd^*(A)\,:\,d\in \ker(\partial)\}.
\]
Now the \emph{Poisson Makar-Limanov$^*_d$ invariant} of $A$ is defined to be
\[{\pml^*_d}(A)=\bigcap\limits_{\partial\in {\plnd^*_d}(A)}\,\ker(\partial) .\]
We say $A$ is \emph{$\plnd^*_d$-rigid} if $A=\pml^*_d(A)$, or equivalently, $\plnd^*_d(A)=\{0\}$. Moreover, we say A is {\it strongly $\plnd^*_d$-rigid} if $A\subseteq \pml^*_d(A[t_1,\ldots,t_n])$, for all $n\ge 1$.

Note that $1\in \ker(\partial)$ for any $\partial\in \plnd^*(A)$. Hence when $d=1$ we will simply omit it and write $\plnd^*$, $\pml^*$, and (strongly) $\plnd^*$-rigid.

\subsection{Filtered and graded Poisson algebras}
\label{sec.filt}

A \emph{Poisson $\NN$-filtration} $F$ on a Poisson algebra $A$ is a collection of $\kk$-subspaces $\{F_i A\}_{i \geq 0}$ satisfying
\begin{enumerate}
    \item $F_i A \subseteq F_{i+1} A$,
    \item $\bigcup_{i\geq 0} F_i A = A$,
    \item $F_i A \cdot F_j A \subseteq F_{i+j} A$, and 
    \item $\{F_i A, F_j A\} \subseteq F_{i+j} A$, for all $i, j \geq 0$.
\end{enumerate}
We say the filtration $F$ is \emph{connected} if $F_0 A = \kk$. 

If $A$ is a Poisson algebra with a Poisson $\NN$-filtration $F$, then the \emph{associated graded Poisson algebra $\gr_F(A)$} is defined as
\[ \gr_F(A) := \bigoplus_{m \geq 0} F_m A/ F_{m-1} A\]
with $F_{-1}(A)=0$. We drop the subscript if the filtration is clear. The Poisson algebra is \emph{$\NN$-graded} if $\gr_F(A) = A$ for some Poisson $\NN$-filtration $F$ and \emph{connected graded} if $F$ is connected. In this setting, we typically write $A_k$ for $F_k A/F_{k-1}A$ so that $A=\bigoplus_{k \geq 0} A_k$. In particular, we call any element $a\in A_k$ a homogeneous element of degree $k$ and write $\deg a=k$.

If $A=\bigoplus_{m\ge 0} A_m$ is an $\NN$-graded Poisson algebra and $d \in \NN$, the $d$th Veronese subalgebra $T=\bigoplus_{m\ge 0} A_{dm}$ is a Poisson subalgebra of $A$ since $\{A_{kd},A_{\ell d}\} \subset A_{(k+\ell)d}$ for all $k,\ell \in \NN$.

A Poisson bracket on $A=\kk[x_1,\hdots,x_n]$ is \emph{quadratic} if $\{x_i,x_j\} \in A_2$ under the standard $\NN$-grading on $A$. Similarly, a Poisson bracket on $A=\kk[x_1,\hdots,x_n]$ is \emph{filtered quadratic} if $\{x_i,x_j\} \in A_{\leq 2}$ under the standard $\NN$-filtration on $A$.

\subsection{Gelfand-Kirillov dimension}

Given an algebra $A$, a \emph{subframe} of $A$ is a finite-dimensional $\kk$-subspace of $A$ containing $1$. We define the \emph{Gelfand-Kirillov dimension} and  \emph{Gelfand-Kirillov transcendence degree} of $A$, respectively, to be
\begin{align*}
\GKdim A &= \sup_V \limsup_{n \to \infty} \log_n \dim(V^n) \\
\Tdeg A &= \sup_V \inf_b \lim_{n \to \infty} \log_n \dim( (\kk + bV)^n )
\end{align*}
where $V$ ranges over all subframes of $A$ and $b$ ranges over the regular elements of A. The dimension of a $\kk$-vector space is always denoted by dim. Note if $A$ is a commutative field then $\Tdeg A =\GKdim A=\Trdeg A$, where $\Trdeg A$ denotes the classical transcendence degree of $A$.

For a commutative affine algebra $A$, the Krull dimension of $A$ ($\Kdim A$) is equal to $\GKdim A$ \cite[Theorem 4.5]{KL}. Hence, for an affine Poisson algebra we frequently use the two concepts interchangeably.

\subsection{Examples} 
\label{sec.exs}

We will consider several families of Poisson algebras which we define below.

\begin{enumerate}
    \item The \emph{$n$-th Weyl Poisson algebra} is the algebra $\cP_n = \kk[x_1,y_1,\hdots,x_n,y_n]$ with Poisson bracket given by
\[ 
\{ x_i, y_j \} = \delta_{ij}, \quad
\{ x_i, x_j \} = \{ y_i,y_j \} = 0,
\]
for all $0 \leq i,j \leq n$. 
In particular, the \emph{first Weyl Poisson algebra} is $\cP_1 = \kk[x,y]$ with bracket $\{x,y\}=1$. We denote the quotient division ring of $\cP_1$ by $\FF_1(\kk)$. 

We can express $\cP_n$ as the iterated Poisson-Ore extensions $\kk[x_1,\hdots,x_n][y_1;\delta_1]_P \cdots [y_n;\delta_n]_P$ in which $\delta_i(x_j)=\delta_{ij}$ for all $i,j$ and $\delta_i(y_j)=0$ for $j<i$. Further, $\cP_n$ is filtered under the standard filtration defined by setting $F_iA$ to be the $\kk$-span of monomials of degree at most $i$. In this case, $\gr_F(\cP_n) = \kk[x_1,y_1,\hdots,x_n,y_n]$ with the trivial Poisson bracket.

\item Let $\Lambda=(\lambda_{ij})\in M_n(\kk)$ be a skew-symmetric $n\times n$ matrix. The \emph{skew-symmetric Poisson algebra} is the polynomial algebra $\kk[x_1,\hdots,x_n]$ with Poisson bracket $\{x_i,x_j\} = \lambda_{ij} x_ix_j$ for all $1\leq i,j\leq n$. The quotient division ring is denoted by $\QQ_n^\Lambda(\kk)$. For simplicity, in the case of $n=2$ we let $\lambda=\lambda_{12}$ and write $\QQ_2^\lambda(\kk)$.
A skew-symmetric Poisson algebra is quadratic and can be presented as the iterated Poisson-Ore extensions $\kk[x_1][x_2;\alpha_2]_P \cdots [x_n;\alpha_n]_P$ in which $\alpha_i(x_j) = \lambda_{ij} x_j$ for $i<j$.

\item Let $\fg$ be an $n$-dimensional Lie algebra over $\kk$ with basis $\{x_1,\hdots,x_n\}$. The \emph{Kostant-Kirillov Poisson bracket} on $S(\fg)=\kk[x_1,\hdots,x_n]$ is determined by $\{x_i,x_j\}=[x_i,x_j]$. We denote this Poisson algebra by $PS(\fg)$.
\end{enumerate}

\section{Cancellation of quadratic Poisson algebras in three variables}
\label{section3}
The invariant theory of quadratic Poisson algebras is similar in many ways to that of quantum polynomial rings (see \cite{GVX,GXW}). This section is inspired by work in \cite{TRZ} in which the authors study cancellation problems for AS regular algebras of global dimension 3. Here we prove an analogous result for Poisson algebras. In particular, we prove that quadratic Poisson algebras in three variables with nontrivial Poisson bracket are Poisson cancellative.

\subsection{Filtered quadratic Poisson algebras in two variables}

In \cite{dumas}, Dumas classified filtered quadratic Poisson algebras on $\kk[x,y]$ up to rational equivalence. Here we classify them up to Poisson isomorphism. This will be useful later in considering cancellation problems for quadratic Poisson algebras with three variables.

For a Poisson algebra $A$, the \emph{commutator ideal} $\{A,A\}$ is the Poisson ideal generated by elements of the form $\{a,b\}$ with $a,b \in A$.

\begin{theorem}
\label{thm.filtered}
Suppose $\kk$ is algebraically closed and $\ch(\kk) \neq 2$.
Let $A=\kk[x,y]$ be a Poisson algebra such that $\{x,y\}=f$ with $f\in A_{\leq 2}$. Then up to a change of variables, the possibilities for $f$ are
\begin{enumerate}
    \item[(1)] $f=0$,
    \item[(2)] $f=1$,
    \item[(3)] $f=x$,
    \item[(4a)] $f=x^2$, (4b) $f=x^2+1$,
    \item[(5a)] $f=\lambda xy$ with $\lambda \in \kk^\times$,
    (5b) $f=\lambda xy+1$ with $\lambda \in \kk^\times$.
\end{enumerate}
Moreover, the Poisson algebras determined by $f$ above are pairwisely nonisomorphic with the exception of replacing $\lambda$ by $-\lambda$ in (5a) and (5b).
\end{theorem}
\begin{proof}
If $f=0$, then we are in case (1). We now assume that $f \neq 0$. Write $f=f_2+f_1+f_0$ where $\deg(f_i)=i$. 

Suppose $f_2=0$. Then $f=ax+by+c$ for $a,b,c \in \kk$. If $a=b=0$, then $c \neq 0$ and we make the substitution $x \mapsto cx'$. This gives case (2).  
After possibly switching $x$ and $y$, we may now assume $a \neq 0$. Making the change of variables $x \mapsto x'-by'-c/a$ and $y\mapsto ay'$ gives (3). 

Now assume $f_2 \neq 0$.
Applying \cite[Lemma 2.5]{GVX}, we may now assume $f_2=x^2$ or $f_2=\lambda xy$ for some $\lambda \in \kk^\times$.

Suppose $f = x^2+ax+by+c$. If $a=b=c=0$, then we are in case (4a). If $a=b=0$ and $c\neq 0$, then we make the change of variable $x \mapsto \sqrt{c} x'$, $y \mapsto \sqrt{c} y'$ and we are in case (4b). If $b = 0$, then making the change of variable $x \mapsto x'-a/2$ and $y \mapsto y'$ leads to case (4a) or (4b). Assume $b \neq 0$ and let $w$ be a root of $w^2+aw+c=1$. Making the change of variable $x \mapsto bx'+w$ and $y \mapsto -(a+2w)x'+by'$ gives $f=x^2+y$. Composing this with the change of variable $x \mapsto -y'$ and $y \mapsto x'-(y')^2$ gives (3).

Suppose $f=\lambda xy+ax+by+c$ for some $\lambda \neq 0$. Making the change of variable $x \mapsto x'-b/\lambda $ and $y\mapsto y'-a/\lambda$ gives $\{x',y'\}=\lambda x'y'+c'$ for some $c'$. If $c' = 0$ we are in case (5a). If $c' \neq 0$, then make the change of variable $x' \mapsto c' x''$, $y' \mapsto y''$ and we are in case (5b).

We now determine the isomorphism classes of these two-generated Poisson algebras by using quotient division rings.

In case (1), $A$ is commutative. By \cite{dumas}, the quotient division rings for the Poisson algebras in (2), (3), (4a), and (4b) are isomorphic to $\FF_1(\kk)$, while those in (5a) and (5b) are isomorphic to $\QQ_2^\lambda(\kk)$. Further, $\FF_1(\kk)$ is not isomorphic to $\QQ_2^\lambda(\kk)$ \cite[Corollary 5.3]{GLdm2}.

Next we show that those in (2), (3), (4a), and (4b) are pairwisely non-isomorphic. In case (2), $A$ is the first Weyl Poisson algebra and hence is a simple Poisson algebra. In case (4b), there are exactly two height one Poisson primes: $(x\pm i)$. In cases (3) and (4a), there is exactly one height one Poisson prime: $(x)$. The commutator ideal $\{A,A\}$ in case (3) is the ideal $(x)$ and in case (4a) it is $(x^2)$. An isomorphism must preserve both the set of height one Poisson prime ideals and the commutator ideal, hence the Poisson algebras in case (3) and (4a) are non-isomorphic.

In case (5a), there are two height one Poisson prime ideal: $(x)$ and $(y)$. In case (5b) there is one height one Poisson prime ideal: $(\lambda xy+1)$. So (5a) and (5b) are pairwisely nonisomorphic. In case (5a), by making the change of variable $x \mapsto y'$ and $y \mapsto x'$, we see that we may replace $\lambda$ by $-\lambda$ up to isomorphism. In case (5b), by making the change of variable $x \mapsto -y'$ and $y \mapsto x'$, we see that we may replace $\lambda$ by $-\lambda$ up to isomorphism. Applying the argument in \cite[p.16]{dumas} completes the proof.
\end{proof}

When $\mathrm{char}(\kk)=0$, the Poisson algebras in Theorem \ref{thm.filtered} (2)-(5) are universally Poisson cancellative by \cite[Corollary 5.5]{GXW}. 

In the next result, we consider \emph{homogenizations} of the Poisson algebras from Theorem \ref{thm.filtered}. This introduces one additional isomorphism class.

\begin{theorem}
\label{thm.homog}
Suppose $\kk$ is algebraically closed and $\ch(\kk) \neq 2$. Let $A=\kk[x,y,t]$ be an $\NN$-graded Poisson algebra with $t$ Poisson central and $\{x,y\} = f\in A_2$. Then up to a change of variables, the possibilities for $f$ are
\begin{enumerate}
    \item[(1)] $f=0$,
    \item[(2)] $f=t^2$,
    \item[(3)] $f=xt$,
    \item[(4a)] $f=x^2$, (4b) $f=x^2+t^2$, (4c) $f=x^2+yt$,
    \item[(5a)] $f=\lambda xy$ with $\lambda \in \kk^\times$,
    (5b) $f=\lambda xy+t^2$ with $\lambda \in \kk^\times$.
\end{enumerate}
Moreover, the Poisson algebras determined by $f$ above are pairwisely nonisomorphic with the exception of replacing $\lambda$ by $-\lambda$ in (5a) and (5b).
\end{theorem}
\begin{proof}
The proof is analogous to that of Theorem \ref{thm.filtered}. The exception is case (4c). It suffices to prove that (4c) is distinct from cases (4a) and (4b). We note in all of these cases except (1) that the Poisson center is $\kk[t]$. 

Let $A=\kk[x,y,t]$ with $t$ Poisson central and $\{x,y\}=x^2+yt$. Let $B=\kk[X,Y,T]$ with $T$ Poisson central and $\{X,Y\}=X^2$. Suppose $\phi:A \to B$ is a Poisson isomorphism. By \cite[Theorem 4.2]{GXW}, we may assume that $\phi$ is a graded Poisson isomorphism. Further, the isomorphism preserves the Poisson center, so $\phi(t)=\alpha T$ for some $\alpha \in \kk^\times$. Hence, $A/(t-1) \iso B/(T-\alpha\inv)$. But this is implies that cases (3) and (4a) of Theorem \ref{thm.filtered} are isomorphic, a contradiction. 
A similar argument shows that cases (4b) and (4c) are distinct.
\end{proof}

The classification in Theorem \ref{thm.filtered} aligns with that of filtered Artin-Schelter regular algebras, while the classification in Theorem \ref{thm.homog} aligns with that of their homogenizations \cite{Gtwogen}. The latter will be important in one of our main results of the section, Theorem \ref{thm.3var}, which shows that any quadratic Poisson algebra in three variables with nontrivial Poisson bracket is cancellative. Before we prove that, however, we first consider some intermediate results. The first is a Poisson analogue of \cite[Lemma 3.2]{TRZ}. Note that we can easily replace ``cancellative'' in the next result with ``$\pcnt$-retractable''.

\begin{lemma}
\label{lem.tensor}
Let $A$ be a Poisson algebra with trivial Poisson center and let $R$ be a Poisson algebra with trivial bracket that is (resp. strongly) cancellative. Then $A \tensor R$ is (resp. strongly) Poisson cancellative.
\end{lemma}
\begin{proof}
Suppose $R$ is strongly cancellative and let $B$ be a Poisson algebra such that there is a Poisson algebra isomorphism
$\phi: (A \tensor R)[t_1,\hdots,t_n] \to B[s_1,\hdots,s_n]$.
Then $\phi$ restricts to a Poisson isomorphism of Poisson centers $\phi_{\pcnt}: R[t_1,\hdots,t_n] \to \pcnt(B)[s_1,\hdots,s_n]$.
However, since the bracket is trivial on both sides, this is just an algebra isomorphism and so $R \iso \pcnt(B)$ by the strong cancellativitiy of $R$.
Let $I$ be the Poisson ideal of $\pcnt(B)[s_1,\hdots,s_n]$ generated by all the $s_i$'s.
Since $\phi\inv(s_i)=\phi_{\pcnt}\inv(s_i) \in R[t_1,\hdots,t_n]$, then $J:=\phi\inv(I)$ is an 
ideal of $R[t_1,\hdots,t_n]$ and 
$A \tensor (R[t_1,\hdots,t_n]/J) \iso B \tensor (\kk[s_1,\hdots,s_n]/I) \iso B$.
Since $\pcnt(A)=\kk$, then by taking the Poisson center of both sides we obtain 
$R[t_1,\hdots,t_n]/J \iso \pcnt(B) \iso R$. Hence
$B \iso A \tensor (R[t_1,\hdots,t_n]/J) \iso A \tensor R$.
\end{proof}

\begin{corollary}
Suppose $\ch(\kk)=0$. Let $A$ be a Poisson algebra with trivial bracket. Let $\delta \in \plnd(A)$ with $\delta(y)=1$ for some $y \in A$.
If $C=\ker(\delta)$ is (resp. strongly) cancellative, then $A[x;\delta]_P$ is (resp. strongly) cancellative.
\end{corollary}
\begin{proof}
By \cite[Lemma 14.6.4]{MR}, $A=C[y]$ and $C \iso A/Ay$. It follows that $A[x;\delta]_P \iso C \tensor \cP_1(\kk)$ as Poisson algebras.
Since $C$ is (resp. strongly) cancellative and $\pcnt(\cP_1(\kk))=\kk$, then the result follows immediately from Lemma \ref{lem.tensor}.
\end{proof}

\begin{lemma}
\label{lem.kdim}
Suppose $\ch(\kk) = 0$.
If $A$ is an affine Poisson domain with $\Kdim A = 3$ and nontrivial bracket, then $\Kdim \pcnt(A) \leq 1$.
\end{lemma}
\begin{proof}
This is essentially a modification of the argument in \cite[Corollary 5.6]{GXW}. Let $\pcnt=\pcnt(A)$, let $F$ be the quotient field of $\pcnt$, and let $A_{\pcnt}$ be the localization of $A$ at $\pcnt$. As $A$ is affine over $\kk$, then $A_{\pcnt}$ is affine over $F$, so $3 = \Kdim A \geq \Kdim_F(A_{\pcnt}) + \Kdim(\pcnt)$
by \cite[Corollary 2]{SmZ}. Since $A$ has nontrivial Poisson bracket, so does $A_{\pcnt}$ and thus $\Kdim(A_{\pcnt}) \geq 2$ \cite[Lemma 13]{AF}. Hence, $\Kdim(\pcnt) \leq 1$. 
\end{proof}

For the case $\Kdim \pcnt(A) = 1$, we have two cases: $\pcnt(A)\iso\kk[t]$ or $\pcnt(A)\niso\kk[t]$. The second of these is easily dealt with.

\begin{lemma}
\label{lem.notkt}
Let $A$ be a noetherian Poisson domain such that $\pcnt=\pcnt(A)$ is connected graded and $\Kdim \pcnt \leq 1$. If $\pcnt \niso \kk[t]$, then $A$ is strongly $\pcnt$-retractable and strongly Poisson cancellative.
\end{lemma}
\begin{proof}
Since $\pcnt$ has trivial Poisson bracket, then we may apply results from the associative setting. 
If $\Kdim \pcnt = 0$, $\pcnt$ is strongly retractable and strongly cancellative by classical results (see \cite[(1.9)]{AEH}). 

Now assume $\Kdim \pcnt = 1$. Applying \cite[Lemma 2.4(1)]{TRZ}, we have that $\pcnt$ is noetherian and finitely generated over $\kk$. By \cite[Corollary 3.4]{AEH}, $\pcnt$ is strongly retractable 
(see also \cite[Lemma 3.1]{TRZ}). Thus, $\pcnt$ is strongly $\lnd$-rigid \cite[Remark 3.7(6)]{LWZ1},
so $A$ is strongly $\pcnt$-retractable \cite[Lemma 7.6(4)]{GXW} and strongly Poisson cancellative \cite[Theorem 7.7(4)]{GXW}. 
\end{proof}

We now attempt to deal with the case $\pcnt(A)=\kk[t]$.
We have the following Poisson analogue of \cite[Lemma 3.5]{TRZ}.

\begin{lemma}
\label{lem.tcnt}
Let $A$ be a noetherian Poisson algebra such that $\pcnt=\pcnt(A)=\kk[t]$ for some $t \in A$.
\begin{enumerate}
\item \label{tcnt1} If $\cP$ is a property such that the Poisson $\cP$-discriminant is $t$, then $A$ is strongly $\pcnt$-retractable and strongly Poisson cancellative.

\item \label{tcnt2} If $t$ is in the Poisson ideal $\{A,A\}$ of $A$ and $\{A,A\} \neq A$, then $A$ is strongly $\pcnt$-retractable and strongly Poisson cancellative.
\end{enumerate}
\end{lemma}

\begin{proof}
\ref{tcnt1} By \cite[Example 2.8]{LWZ1}, $t$ is \emph{effective} (we refer the reader to that reference for the definition of effective). Thus, $\pcnt$ is strongly $\plnd_t$-rigid \cite[Lemma 7.15(2)]{GXW}, so $A$ is strongly $\pcnt$-retractable \cite[Lemma 7.6(4)]{GXW}. Furthermore, $A$ is strongly Poisson cancellative \cite[Theorem 7.16(2)]{GXW}.

\ref{tcnt2} 
Let $\cP$ be the property of a Poisson algebra $B$ that $\{B,B\}=B$. We claim the $\cP$-discriminant of $A$ is $t$. By hypothesis, $(t) \in \Maxspec(\pcnt)$ and $t \in \{A,A\}$. Set $\overline{A}=A/(t)$, then $\{\overline{A},\overline{A}\}\neq \overline{A}$. Hence, $(t) \in D_\cP(A)$.
On the other hand, for $\alpha \in \kk^\times$, $(t-\alpha) \notin \{A,A\}$ and so by maximality, 
$(t-\alpha) + \{A,A\} = A$. Hence, $(t-\alpha) \notin D_\cP(A)$. Thus, the $\cP$-discriminant is $t$. 
\end{proof}

The following result is the Poisson version of \cite[Lemma 2.6]{TRZ}.

\begin{lemma}
\label{lem.gldim}
Let $A$ be a connected graded Poisson domain and let $t \in \pcnt(A)$ be a homogeneous element of positive degree $d$. For every $\alpha \in \kk^\times$, $A/(t-\alpha)$ contains a copy of $(A[t\inv])_0$ as a Poisson subalgebra. Moreover, suppose that $A$ is generated in degree 1 and $d\neq 0$ in $\kk$. Then 
\[ \gldim A/(t-\alpha)= \gldim (A[t\inv])_0,\]
which is finite if $\gldim A<\infty$.
\end{lemma}
\begin{proof}
Let $I$ be the principal ideal of $A$ generated by $t-\alpha$, and let $a(t-\alpha) \in I$. Since $t-\alpha \in \pcnt(A)$, then for any $b \in A$,
\[ \{ b, a(t-\alpha) \} = \{ b, a \}(t-\alpha) + \{ b, t-\alpha\} a = \{b,a\}(t-\alpha) \in I.\]
That is, $I$ is a Poisson ideal of $A$.

Let $d=\deg t$ and let $T=\bigoplus_{\ell\ge 0} A_{d\ell}$ be the $d$th Veronese subalgebra of $A=\bigoplus_{\ell\ge 0} A_\ell$. Hence we have natural isomorphisms 
\[
T/(t-\alpha)=T/(\alpha^{-1}t-1)\cong (T[(\alpha^{-1}t)^{-1}])_0 \cong (T[t^{-1}])_0\cong (A[t^{-1}])_0
\]
where each isomorphism appearing above is a Poisson isomorphism since $t \in \pcnt(A)$. The result follows via the embedding $T/(t-\alpha)\hookrightarrow A/(t-\alpha)$. The rest of the argument regarding the global dimension is identical to \cite[Lemma 2.6(2),(3)]{TRZ}. 
\end{proof}

The following results follow directly from \cite[Lemma 3.7 (2),(3)]{TRZ}, with a substitution of Lemma \ref{lem.gldim} for \cite[Lemma 2.6]{TRZ} where necessary.

\begin{lemma}
\label{lem.kt}
Suppose $\ch(\kk)=0$.
Let $A$ be a noetherian connected graded Poisson domain that is generated in degree 1. Assume $\pcnt=\pcnt(A)=\kk[t]$ for some homogeneous element $t \in A$ of positive degree, and $\gldim A/(t) = \infty$. If $\gldim A/(t-1) < \infty$ or $\gldim A < \infty$, then $A$ is strongly $\pcnt$-retractable and strongly Poisson cancellative.
\end{lemma}

\begin{theorem}
Suppose $\ch(\kk)= 0$. Let $A=\kk[x_1,\hdots,x_n]$ be a Poisson algebra with the standard (Poisson) grading and $\pcnt=\pcnt(A)$. Suppose that $\Kdim \pcnt \leq 1$ and $\gldim A/(t)=\infty$ for every homogeneous element $t \in \pcnt$ of positive degree. Then $A$ is Poisson cancellative.
\end{theorem}
\begin{proof}
If $\pcnt \niso \kk[t]$, then the result follows from Lemma \ref{lem.notkt}. Otherwise the result follows from Lemma \ref{lem.kt}.
\end{proof}

Note that the following theorem gives an affirmative answer to \cite[Question 8.4]{GXW} asking whether any polynomial Poisson algebra in three variables with nontrivial Jacobian bracket is Poisson cancellative?

\begin{theorem}
\label{thm.3var}
Suppose $\kk$ is algebraically closed and $\ch(\kk)=0$. Let $A$ be a quadratic polynomial Poisson algebra with three variables. If $A$ has nontrivial bracket, then $A$ is strongly Poisson cancellative.
\end{theorem}
\begin{proof}
Set $\pcnt=\pcnt(A)$. By Lemma \ref{lem.kdim}, $\Kdim \pcnt \leq 1$. If $\pcnt \niso \kk[t]$, then the result follows from Lemma \ref{lem.notkt}. 

We now assume $\pcnt \iso \kk[t]$. If $\gldim A/(t) = \infty$, then we can apply Lemma \ref{lem.kt}, so suppose $\gldim A/(t) < \infty$. Since $A/(t)$ is a regular graded algebra, it is again a polynomial algebra. By an easy Hilbert series argument, we can write $A=\kk[x,y,t]$ with $t \in \pcnt$ and $\{x,y\}=f \in A_2$. That is, $A$ is one of the algebras from Theorem \ref{thm.homog} and we can consider each case:
\begin{enumerate}
    \item[(2)] $f=t^2$. See \cite[Example 7.2]{GXW}.
    \item[(3)] $f=xt$. Let $\cP$ be the property that $A/\fm_\alpha A$ has trivial Poisson bracket, \[\fm_\alpha = (t-\alpha) \in \Maxspec(\pcnt).\] Then the $\cP$-discriminant is $t$. Hence, $A$ is strongly Poisson cancellative by Lemma \ref{lem.tcnt}.
    \item[(4a)] $f=x^2$. In this case, $A \iso A' \tensor \kk[t]$ where $A'=\kk[x,y]$ with Poisson bracket $\{x,y\}=x^2$. The result follows from Lemma \ref{lem.tensor}.
    \item[(4b)] $f=x^2+t^2$. Let $\cP$ be the property that $A/\fm_\alpha A$ is isomorphic to $A/\fm_0 A$, with $\fm_\alpha$ as in (3). Then the $\cP$-discriminant is $t$. Hence, $A$ is strongly Poisson cancellative by Lemma \ref{lem.tcnt}.
\end{enumerate}
Cases (4c) and (5b) are similar to (4b), while (5a) is similar to (4a).
\end{proof}

\subsection{Veronese Poisson algebras.}
Recall from Section \ref{sec.filt} that Veronese subalgebras of $\NN$-graded Poisson algebras are Poisson subalgebras. Here we prove a cancellation result for Veronese subalgebras of quadratic polynomial Poisson algebras in three variables.

Let $A$ be a graded commutative algebra generated in degree one. Then $A$ is a \emph{(graded) isolated singularity} if $\gldim A = \infty$ but $A_{(\fp)}$ is regular for any homogeneous prime ideal $\fp \neq A_{\geq 1}$. We also refer to \cite{ueyama} for a suitable notion for noncommutative algebras.

\begin{lemma}
\label{lem.ver}
Suppose $\ch(\kk)=0$. Let $A$ be a connected graded Poisson domain that is generated in degree one, and let $t \in \pcnt(A)$ be a homogeneous element of positive degree $d$. Then for every $\alpha \in \kk$, $\alpha \neq 0$, $A/(t-\alpha)$ has finite global dimension.
\end{lemma}
\begin{proof}
It suffices by Lemma \ref{lem.gldim} to prove that $(A[t\inv])_0$ has positive global dimension. The rest of the proof follows \cite[Lemma 4.2]{TRZ}.
\end{proof}

\begin{theorem}
\label{thm.ver}
Suppose $\ch(\kk)=0$. Let $A$ be a noetherian connected graded Poisson domain generated in degree 1. If $A$ is a graded isolated singularity and $\Kdim \pcnt(A) \leq 1$, then $A$ is strongly Poisson cancellative.
\end{theorem}
\begin{proof}
If $\pcnt(A) \niso \kk[t]$, then the result follows from Lemma \ref{lem.notkt}. Now suppose $\pcnt(A)\iso \kk[t]$. Note that we can choose $t$ such that it is homogeneous of positive degree. Since $A$ is a graded isolated singularity, $\gldim A = \infty$ and hence $\gldim A/(t) = \infty$ as well \cite[Lemma 7.6]{LPWZ}.
Lemma \ref{lem.ver} now implies that $A/(t-\alpha)$ has finite global dimension for every nonzero $\alpha \in \kk$. The result now follows from Lemma \ref{lem.kt}.
\end{proof}

\begin{corollary}\label{cor:veronese}
Suppose $\ch(\kk)=0$. Let $A$ be a quadratic polynomial Poisson algebra with three variables. If $A$ has nontrivial bracket, then the $d$th Veronese algebra $A^{(d)}$ is Poisson cancellative for every $d \geq 1$.
\end{corollary}
\begin{proof}
It follows from the same argument as in \cite[Corollary 0.5]{TRZ} that $A^{(d)}$ is a graded isolated singularity. Note that $\Kdim A^{(d)} = \Kdim A$, so it remains only to show that the bracket on $A^{(d)}$ is nontrivial. Then Lemma \ref{lem.kdim} implies that $\Kdim \pcnt(A^{(d)}) \leq 1$ so the result follows from Theorem \ref{thm.ver}.

Suppose $A^{(d)}$ had trivial Poisson bracket. Since $A$ is generated in degree one, and hence strongly graded, then it follows that $A$ is finitely generated over $A^{(d)}$. But then $A$ would have trivial Poisson bracket due to \cite[Proposition 1.4 (2)]{CAP2013}, a contradiction.
\end{proof}

\subsection{Kostant-Kirillov brackets}
We next apply our methods above to study cancellation problems for Poisson algebras equipped with Konstant-Kirillov bracket. We will show that, for $\dim \fg \leq 3$, the Poisson cancellation properties of $PS(\fg)$ mirror the cancellation properties of $U(\fg)$. That is, our analysis will closely follow \cite[Example 3.10]{TRZ}.

\begin{theorem}
\label{thm:Lie}
Suppose $\ch(\kk) = 0$.
Let $\fg$ be a non-abelian Lie algebra of dimension $\leq 3$. Then $PS(\fg)$ is strongly Poisson cancellative.
\end{theorem}

\begin{proof}
Let $\dim \fg =2$. Since $\fg$ is non-abelian, then $\pcnt(PS(\fg))=\kk$. The result follows from \cite[Theorem 5.5]{GXW}.

Now suppose that $\dim \fg = 3$. We use Bianchi's classification to break this into various isomorphism classes \cite[Section 1.4]{JaLIE}.

(1) Suppose $\fg=\fsl_2$. See \cite[Example 7.14]{GXW}.

(2) Suppose $\fg$ is the Heisenberg Lie algebra with basis $\{e,f,g\}$ and relations $\{e,f\}=g$, $\{g,-\}=0$.
In this case, $\pcnt(A)=\kk[g]$. Moreover $g \in \{A,A\}$ and $\{A,A\} \neq A$. Hence we can apply Lemma \ref{lem.tcnt}.

(3) Suppose $\fg=L \oplus \kk z$ where $L$ is the two-dimensional non-abelian Lie algebra. Then $\pcnt(PS(L))=\kk$ and so $PS(L)$ is cancellative by \cite[Theorem 5.5]{GXW}. We can now apply Lemma \ref{lem.tensor}.

(4) Suppose $\fg$ has basis $\{e,f,g\}$ and relations
\[ [e,f]=0, \quad [e,g]=e, \quad [f,g]=\alpha f, \alpha \neq 0.\]

In this case, $A=PS(\fg)$ is the Poisson-Ore extension $\kk[e,f][g;\delta]_P$ in which the Poisson bracket on $\kk[e,f]$ is trivial and $\delta$ is a Poisson derivation of $\kk[e,f]$ given by $\delta(e)=-e$ and $\delta(f)=-\alpha f$.
Set $\pcnt=\pcnt(A)$. Then
$\pcnt=\{ x \in \kk[e,f] : \delta(x) = 0\}$, which is a graded subring of $\kk[e,f]$. 

If $\pcnt=\kk$, then $A$ is Poisson cancellative by \cite[Theorem 5.5]{GXW}. Suppose $\pcnt \neq \kk$, so $\Kdim \pcnt \geq 1$. By Lemma \ref{lem.kdim}, $\Kdim \pcnt \leq 1$ so in fact $\Kdim \pcnt = 1$. Then $\pcnt$ is a noetherian graded domain which is finitely generated over $\kk$ \cite[Lemma 2.4(1)]{TRZ}. If $\pcnt \niso \kk[t]$, then $A$ is strongly Poisson cancellative by Lemma \ref{lem.notkt}.

Finally, suppose $\pcnt=\kk[t]$ for some homogeneous $t \in \kk[e,f]$.
Then $\deg t \geq 2$, so $\kk[e,f]/(t)$ has infinite global dimension while $\kk[e,f]/(t-\alpha)$ has finite global dimension for all $\alpha \in \kk^\times$ \cite[Lemma 2.6(3)]{TRZ}. Then for $\alpha \in \kk$,
$A/(t-\alpha) = (\kk[e,f]/(t-\alpha))[g]$ has finite global dimension if and only if $\alpha \in \kk^\times$. Thus, if $\cP$ is the property of having infinite global dimension, then the $\cP$-discriminant is $t$ and so $A$ is strongly Poisson cancellative by Lemma \ref{lem.tcnt}.

(5) Suppose $\fg$ has basis $\{e,f,g\}$ and relations
\[ [e,f]=0, \quad [e,g]=e+\beta f, \quad [f,g]= f, \beta \neq 0.\]
This is similar to (4).
\end{proof}

See Section \ref{sec.questions} for further remarks concerning cancellation of Poisson algebras $PS(\fg)$.

\subsection{Poisson cancellation of non-affine algebras}

\begin{lemma}
Assume $A$ and $B$ are Poisson algebras. Let $A'=A[s_1,...,s_n]$, $B'=B[t_1,...,t_n]$, and let $\phi:A'\rightarrow B'$ be an isomorphism of Poisson algebras.
\begin{enumerate}
    \item \label{cancel1} If $\phi(A)$ is a Poisson subalgebra of $B$, then $\phi(A)=B$.
    \item \label{cancel2} If $\phi(\pcnt(A))\subseteq B$, then $A\cong B$.
\end{enumerate}
\end{lemma}
\begin{proof} We follow the proof of \cite[Lemma 5.1]{TZZ}.

\ref{cancel1} By the given Poisson isomorphism, $\pcnt(A')\cong \pcnt(B')$. Additionally, we have that $\pcnt(B)=B\cap \pcnt(B')$, and since $[s_i,A]=0$ and $[t_i,B]=0$ for all $i\leq n$, then we obtain an isomorphism between algebras\begin{align*}
\phi:\pcnt(A)[s_1,...,s_n]\rightarrow \pcnt(B)[t_1,...,t_n].
\end{align*}
Since $\phi(A)\subseteq B$, then $\phi(\pcnt(A))\subseteq \pcnt(B)$. By \cite[Lemma 2]{BR72}, we obtain $\phi(\pcnt(A))=\pcnt(B)$.

Now set $f_i=\phi(s_i)$ for all $i\leq n$. The isomorphism implies that
\begin{align*}
\pcnt(B)[t_1,...,t_n]=\pcnt(B)[f_1,...,f_n].
\end{align*}
Using \cite[Lemma 1.2]{TZZ}, we have $B'=B[f_1,...,f_n]$. Define $\tau:B'\rightarrow B[f_1,...,f_n]$ as a $B$-automorphism with $\tau(f_i)=t_i$. Then $\tau\circ\phi:A'\rightarrow B'$ is an automorphism of Poisson algebras with $(\tau\circ\phi)(s_i)=t_i$ for all $i\leq n$, and $(\tau\circ\phi)(A)\subseteq B$. Then $(\tau\circ\phi)(A)=B$. Applying $\tau^{-1}$ to both sides implies that $\phi(A)=B$.

\ref{cancel2} Again we have the isomorphism $\pcnt(A)[s_1,...,s_n]\cong \pcnt(B)[t_1,...,t_n]$ given by $\phi$. Thus, $\phi(\pcnt(A))\subseteq \pcnt(B)$, and so equality holds by \cite[Lemma 2]{BR72}. Setting $f_i=\phi(s_i)$ for all $i\leq n$, it follows that $B'=B[f_1,...,f_n]$. Given the ideals $S=(s_1,...,s_n)A'$ and $F=(f_1,..,f_n)B'$, then we have $\phi(S)=F$. It follows that $A\cong A'/S\cong B'/F\cong B$.
\end{proof}

\subsection{Divisor subalgebras}
\label{sec.divisor}

Our final result of this section requires some setup. 
We will first define a Poisson version of the divisor subalgebra invariant, first introduced in \cite{CYZ1}.
Let $A$ be a Poisson domain containing $\ZZ$. 
If $f \in A$ is nonzero, then we define the \emph{set of subwords of $f$} as
\[ \SW(f) = \{ g \in A : f = gb \text{ for some } b \in A\}.\]
Let $F \subset A$ be a nonempty, nonzero subset of $A$. Then $\SW(F)$ is the union of $\SW(f)$ for $0 \neq f \in F$. Set $D_0(F)=F$ and for $n \geq 1$ inductively define $D_n(F)$ as the Poisson $\kk$-subalgebra of $A$ generated by $\SW(D_{n-1}(F))$. The subalgebra $\DD(F) = \bigcup_{n \geq 0} D_n(F)$ is the \emph{$F$-divisor Poisson subalgebra of $A$}. When $F=\{f\}$ we simply write $\DD(f)$ in place of $\DD(F)$. The next result is completely analogous to \cite[Lemma 7.5]{CYZ1}.

\begin{lemma}
Let $F$ be a subset of $\pml(A)$. Then $\DD(F) \subseteq \pml(A)$.
\end{lemma}
\begin{proof}
By \cite[Lemma 7.4]{CYZ1}, any locally nilpotent derivation on $A$ is factorially closed. That is, if $\delta \in \lnd(A)$, then $\delta(ab)=0$ for some $a,b\in A$ implies that $\delta(a)=\delta(b)=0$. 

Now let $f \in F$ and write $f=gb$ for some $g \in A$. Then for any $\delta \in \plnd(A)$, $\delta(f)=0$ implies that $\delta(g)=\delta(b)=0$. Since $\SW(F)$ is the set of all such possible $g$, we get $\SW(F)\subseteq \pml(A)$. One can easily check that $\pml(A)$ is closed under multiplication and Poisson bracket of $A$. Since $D_1(F)$ is the Poisson subalgebra of $A$ generated by all $\SW(F)$, we get $D_1(F)\subseteq \pml(A)$. The result now follows from induction.
\end{proof}

Finally, we conclude with a Poisson analogue of \cite[Proposition 5.2(1)]{TZZ}.

\begin{proposition}
Let $A$ be a Poisson algebra such that $\DD(1) \supseteq \pcnt(A)$. Then $A$ is strongly Poisson cancellative.
\end{proposition}
\begin{proof}
Suppose $\phi:A'=A[s_1,...,s_n]\rightarrow B'=B[t_1,...,t_n]$ is an isomorphism of Poisson algebras with $[s_i,A]=[t_i,B]=0$ for all $i\leq n$. Using \cite[Lemma 4.4]{TZZ}, we have that $\phi(\DD(1))=\phi(\DD_{A'}(1))\subseteq\DD_{B'}(1)\subseteq B$. The assumption $\pcnt(A)\subseteq\DD(1)$ implies that $\phi(\pcnt(A))\subseteq B$ and now apply Lemma 5.1.
\end{proof}

\section{Skew Poisson cancellation}
\label{section4}
In this section, we study skew Poisson cancellation. This section is divided into three parts. First, we study general results which give some conditions for $\alpha$- or $\delta$-cancellativity. In \ref{sec.divisor}, we consider a Poisson analogue of the divisor subalgebra leading to a criteria for general skew cancellation. Finally, in \ref{sec.strat}, we define a version of \emph{stratiform} (ala Schofield \cite{Sc}) in order to further study skew cancellation for certain Poisson algebras.

The next lemmas follows directly from \cite[Lemma 2.2]{TZZ} and  \cite[Proposition 2.3]{TZZ}, respectively, because we can forget the Poisson structure. 

\begin{lemma}
\label{lem.subalg2}
Let $Y$ be a Poisson algebra with a Poisson $\NN$-filtration $\{F_i Y\}_{i \geq 0}$ such that $\gr Y$ is an $\NN$-graded Poisson domain. Suppose $Z$ is a Poisson subalgebra of $Y$ and set $Z_0 = Z \cap F_0Y$.
If $\Kdim Z = \Kdim Z_0 < \infty$, then $Z=Z_0$.
\end{lemma}

\begin{lemma}
\label{lem.subalg}
Let $A$ be a Poisson domain and let $Y=A[t_1;\alpha_1,\delta_1]_P\cdots[t_n;\alpha_n,\delta_n]_P$ be an iterated Poisson-Ore extension. If $B$ is a Poisson subalgebra of $Y$ containing $A$ such that $\Kdim A = \Kdim B < \infty$, then $A=B$.
\end{lemma}

The following results are analogues of \cite[Lemma 4.5]{TZZ} and \cite[Theorem 4.6]{TZZ}, respectively.

\begin{lemma}
\label{lem.scalar}
Let $D=B[t_1;\alpha_1]_P\cdots[t_n;\alpha_n]_P$ be an iterated-Poisson-Ore extension over a Poisson algebra $B$. Let $A$ be a factor of $D$ by some  Poisson ideal such that $A$ is Poisson simple and $A^\times = \kk^\times$.
Denote the quotient map by $\pi:D \to A$.
Then the image $\pi(t_i)$ of each $t_i$ in $A$ is a scalar and $\pi(B)=A$. Furthermore, if $B$ is Poisson simple, then $A \iso B$.
\end{lemma}
\begin{proof}
This is trivial when $n=0$. So assume $n>0$. As $t_n$ is Poisson normal in $D$, then $\pi(t_n)$ is Poisson normal in $A$. By simplicity, $\pi(t_n)$ is zero or invertible in $A$. As $A^\times=\kk^\times$, then $\pi(t_n)$ is a scalar in $\kk$. Thus, $\pi(B[t_1;\alpha_1]_P\cdots[t_{n-1};\alpha_{n-1}]_P)=A$. By induction we find $\pi(t_i)$ is a scalar for each $i=1,\hdots,n$ and $\pi(B)=A$. Finally, if $B$ is Poisson simple then $\pi|_B: B\to A$ is both surjective and injective. Hence, $A\cong B$. 
\end{proof}

We now give two sets of conditions for when a Poisson algebra may be strongly Poisson $\alpha$-cancellative.

\begin{theorem}
\label{thm.simple}
Let $A$ be a noetherian Poisson domain of finite Krull dimension. Suppose either
\begin{enumerate}
    \item \label{alpha1} $A$ is Poisson simple and $A^\times = \kk^\times$, or
    \item \label{alpha2} $A$ is affine of Krull dimension 1.
\end{enumerate}
Then $A$ is strongly Poisson $\alpha$-cancellative.
\end{theorem}
\begin{proof}
Let $B$ be a Poisson algebra, set $C=A[t_1;\alpha_1]_P\cdots[t_n;\alpha_n]_P$ and $D=B[s_1;\beta_1]_P\cdots[s_n;\beta_n]_P$ to be Poisson-Ore extensions of $A$ and $B$, respectively, and suppose $\phi:C \to D$ is a Poisson isomorphism. We claim $A \iso B$. Since $A$ is noetherian, then so is $C$. Hence, $B$ and $D$ are noetherian as well. As 
\begin{align}
\label{eq.kdim}
\Kdim(A) + n = \Kdim(C) = \Kdim(D) = \Kdim(B)+n,
\end{align}
then $\Kdim(A) = \Kdim(B)$. 

\ref{alpha1} Let $I$ be the Poisson ideal of $C$ generated by all $t_i$'s so that $A \iso C/I$. Thus, $\phi$ induces a Poisson isomorphism $\overline{\phi}:A \xrightarrow{\sim} D/\phi(I)$, and hence $A\iso D/\phi(I)$ is Poisson simple. 
Set $\pi:D \to D/\phi(I)$ to be the canonical quotient map. By Lemma \ref{lem.scalar}, we get $\pi(s_i)\in \kk$ for all $1\le i\le n$. Thus, the composition $(\overline{\phi})\inv \circ \restr{\pi}{B}$ is a surjective Poisson algebra homomorphism $B \to A$. But since $B$ is a domain with $\Kdim B = \Kdim A$, $(\overline{\phi})\inv \circ \restr{\pi}{B}$ is a Poisson isomorphism.

\ref{alpha2} In this case we have $\Kdim A = \Kdim B = 1$.
Since $\phi$ restricts to an isomorphism of commutative algebras, we have $A[x_1,\hdots,x_n] \iso B[y_1,\hdots,y_n]$. Then $A \iso B$ as commutative algebras by \cite[Corollary 3.4]{AEH}. But by \cite[Lemma 13]{AF}, the brackets on $A$ and $B$ are  trivial. Hence, $A \iso B$ as Poisson algebras.
\end{proof}

We now consider the $\delta$-cancellative property. Here we ``Poissonify'' some results from \cite{BHHV} and \cite{TZZ} on skew cancellation. 
In particular, Theorem \ref{thm.delta} is adapted from \cite[Lemma 5.3]{BHHV} (see also \cite[Lemma 6.11 (1)]{GXW} and \cite[Lemma 3.5]{BZ2}) and \cite[Theorem 5.4]{TZZ}.

\begin{theorem}
\label{thm.delta}
Suppose $\ch(\kk)=0$. Let $A$ be an affine Poisson domain such that $\pml(A)=A$. Then 
\begin{enumerate}
    \item \label{delta1} $\pml(A[x;\delta]_P) = \pml(A)$, and
    \item \label{delta2} $A$ is Poisson $\delta$-cancellative.
\end{enumerate}
\end{theorem}
\begin{proof}
\ref{delta1} Let $\mu$ be the derivation of $A[x]$ such that $\mu(A)=0$ and $\mu(x)=1$. Then it is easy to check that $\mu$ is indeed a locally nilpotent Poisson derivation of $A[x;\delta]_P$. Since $\ker(\mu)=A$, we have $\pml(A[x;\delta]_P)\subseteq \ker(\mu)=A=\pml(A)$.

Conversely, let $\eta$ be a locally nilpotent Poisson derivation of $A[x;\delta]_P$ such that $\eta(\pml(A)) \neq 0$. We have assumed $A$ is finitely generated and so there is some $m \geq 0$ minimal such that $\eta(A) \subseteq A + Ax + \cdots + Ax^m$. Thus, for $a \in A$,
\[\eta(a) = \partial(a)x^m + \text{ lower degree terms }\]
where $\partial$ is a Poisson derivation of $A$ such that $\partial(\pml(A))\neq 0$. If $m=0$, then $\partial$ is a locally nilpotent derivation of $A$ and so vanishes on $\pml(A)$, a contradiction. Then $m>0$ and the proof follows by the same argument in \cite[Lemma 6.11 (1)]{GXW}.

\ref{delta2} Suppose that $\phi:A[x;\delta]_P \to B[y;\delta']_P$ is a Poisson algebra isomorphism for some Poisson algebra $B$. By \eqref{eq.kdim}, $\Kdim A = \Kdim B$. By part \ref{delta1} and our hypothesis,
\[ A = \pml(A) = \pml(A[x;\delta]_P) \iso \pml(B[y;\delta']_P) \subseteq B.\]
That is, $\phi(A) \subset B$. Set $B'=\phi\inv(B)$. Hence, the result follows from Lemma \ref{lem.subalg}.
\end{proof}

\begin{example}
Suppose $\ch(\kk)=0$. By \cite[Theorem 5.5]{GXW}, the first Weyl Poisson algebra $\cP_1$ is universally Poisson cancellative. Moreover, by Theorem \ref{thm.simple}, $\cP_1$ is strongly Poisson $\alpha$-cancellative. Let $B=\kk[y,z]$ with trivial Poisson bracket.
Let $\delta$ be the trivial (Poisson) derivation on $\cP_1$ and let $\delta'$ be the Poisson derivation on $B$ given by $\delta'(y)=1$ and $\delta'(z)=0$. Then $\cP_1[z;\delta]_P \iso B[x;\delta']_P$ as Poisson algebras. Since $\cP_1 \niso B$, then $\cP_1$ is not Poisson $\delta$-cancellative.
\end{example}

The following results can be proved almost immediately from \cite[Lemma 4.4]{TZZ}.
\begin{lemma}
\label{lem.dd}
Let $A,B$ be Poisson algebras. Let $C$ and $D$ be iterated Poisson-Ore extensions of $A$ and $B$, respectively.
\begin{enumerate}
\item For any nonempty, nonzero subset $F$ of $A$, we have $\DD_C(F) = \DD_A(F)$. 
\item If $\phi:C \to D$ is a Poisson isomorphism, then $\phi(\DD_C(1)) = \DD_D(1) \subseteq B$.
\end{enumerate}
\end{lemma}

\begin{theorem}
\label{thm.dd1}
Let $A$ be a noetherian Poisson domain with $\Kdim A < \infty$ and $\DD(1)=A$. Then $A$ is strongly Poisson skew cancellative.
\end{theorem}
\begin{proof}
Let $B$ be a Poisson algebra. Set $C=A[t_1;\alpha_1]_P\cdots[t_n;\alpha_n]_P$ and $D=B[s_1;\beta_1]_P\cdots[s_n;\beta_n]_P$ to be Poisson-Ore extensions of $A$ and $B$, respectively, and suppose $\phi:C \to D$ is a Poisson isomorphism. We claim $A \iso B$. By \eqref{eq.kdim}, $\Kdim A = \Kdim B$. By Lemma \ref{lem.dd} and the hypothesis,
\[ \phi(A) = \phi(\DD_C(1)) = \DD_D(1) \subseteq B.\]
Set $B'=\phi\inv(B)$, so $A \subseteq B'$. Since $\Kdim B = \Kdim B'$, then $\Kdim A = \Kdim B'$.
Applying Lemma \ref{lem.subalg} gives $A=B'$ and hence $A\cong B$ as Poisson algebras. 
\end{proof}

\begin{example}
Let $A=\kk[x_1^{\pm 1},\hdots,x_n^{\pm 1}]$ be the Laurent polynomial ring in $n$ variables
with Poisson bracket $\{x_i,x_j\}=\lambda_{ij} x_ix_j$ for all $i < j$, $\lambda_{ij} \in \kk$.
Then clearly $\DD(1)=A$. So by Theorem \ref{thm.dd1}, $A$ is strongly Poisson skew cancellative.
\end{example}

\subsection{Poisson stratiform algebras}
\label{sec.strat}

If $A$ is affine, then $\Kdim A = \Trdeg Q(A)$ and so we may restate Lemma \ref{lem.subalg} with the equivalent hypothesis that $\Trdeg Q(A)=\Trdeg Q(B)$. We would like a version of this theorem that weakens the affine hypothesis. For this, we will need the following notion that is adapted from \cite{Sc}.

\begin{definition}
\label{def.strat}
Let $S$ be a artinian Poisson algebra. We say that $S$ is \emph{Poisson stratiform} over $\kk$ if there is a chain of artinian Poisson Poisson algebras 
\begin{align}
\label{eq.strat}
S = S_n \supseteq S_{n-1} \supseteq \cdots \supseteq S_1 \supseteq S_0 = \kk
\end{align}
where, for every $i$, either
\begin{enumerate}
    \item[(i)] $S_{i+1}$ is finite over $S_i$; or
    \item[(ii)] $S_{i+1}=S_i(t_i;\alpha_i,\delta_i)_P$
    for an appropriate choice of $\alpha_i,\delta_i$.
\end{enumerate}
A Poisson domain $A$ is said to be \emph{Poisson stratiform} if $Q(A)$ is Poisson stratiform.
\end{definition}

We define the \emph{Poisson stratiform length} of $S$ to be the number of extensions of type (ii) in the chain. Our definition does not immediately imply that $S$ is stratiform in the usual (associative) sense \cite{Sc}. However, we can recover this using results from \cite{Ztdeg}.

\begin{example}
(1) Let $A=\kk[x,y]/(x^2,y^2)$ with Poisson bracket $\{x,y\}=xy$. Then $A$ is Poisson stratiform with Poisson stratiform length 0.

(2) Keep $A$ as is (1) and let $B=A[t;\alpha]_P$ where $\alpha(x)=x$ and $\alpha(y)=y$. Then $B$ is Poisson stratiform with Poisson stratiform length 1.

(3) The $n$th Weyl Poisson algebra $\cP_n$ is Poisson stratiform, as is any skew-symmetric Poisson algebra on $\kk[x_1,\hdots,x_n]$, with Poisson stratiform length $n$.
\end{example}

An algebra $A$ is \emph{$\Tdeg$-stable} if $\Tdeg(A)=\GKdim(A)$ and $\Tdeg(S\inv A)=\Tdeg(A)$ for any Ore set $S$ \cite[Definition 3.4]{Ztdeg}. A commutative algebra is necessarily $\Tdeg$-stable \cite[Proposition 2.2]{Ztdeg} and \cite[4.2]{NVO}.

\begin{proposition}
\label{prop.strat}
Let $S$ be a artinian Poisson simple Poisson algebra that is Poisson stratiform over $\kk$. Then the stratiform length of $S$ is independent of the filtration as in Definition \ref{def.strat}.
\end{proposition}
\begin{proof}
We will show that the number of extensions of the form (ii) is equal to $\Tdeg(S)$ as defined in \cite{Ztdeg}. Of course, $\Tdeg(S_0)=\Tdeg(\kk)=\GKdim(\kk)=0$. Now suppose that for some $k\geq 0$, the stratiform length of $S_k$ is $d=\Tdeg(S_k)$. Suppose $S_{k+1}$ is an extension of type (i). Since $S_{k+1}$ is finite over $S_k$, then by $\Tdeg$-stability and properties of GK dimension,
\[ \Tdeg(S_{k+1})=\GKdim(S_{k+1}) = \GKdim(S_k) = \Tdeg(S_k)= d.\]

Now suppose that $S_{k+1}$ is an extension of type (ii). Then again by $\Tdeg$-stability and properties of GK dimension,
\[ \Tdeg(S_{k+1})
= \Tdeg(S_k(t_k))
= \Tdeg(S_k[t_k]) 
= \GKdim(S_k[t_k])
= \GKdim(S_k) + 1
= \Tdeg(S_k)+1 = d+1.\]
It follows from induction that $\Tdeg$ is equal to the number of extensions of type (ii). Since $\Tdeg$ is an invariant of $S$ and does not depend on the particular filtration, then $\Tdeg$ is equal to the stratiform length of $S$.
\end{proof}

\begin{corollary}
\label{cor.strat}
Let $A$ be Poisson stratiform. Let $B$ be an $n$-step Poisson-Ore extension of $A$. Then $\Tdeg(B) = \Tdeg(A)+n$.
\end{corollary}

This is a Poisson version of \cite[Proposition 2.6]{TZZ}.

\begin{proposition}
\label{prop.strat1}
Let $A$ be a noetherian Poisson stratiform domain. Let $Y=A[t_1;\alpha_1,\delta_1]_P \cdots [t_n;\alpha_n,\delta_n]_P$.
Let $B$ be a Poisson subalgebra of $Y$ containing $A$ that is stratiform.
If $\Tdeg(B)=\Tdeg(A)$, then $A=B$.
\end{proposition}
\begin{proof}
Suppose $A \neq B$. Hence, there exists $a \in B$ such that the sum $Q(A) + aQ(A) + a^2Q(A) + \cdots$ is direct. This implies that $\dim_{Q(B)} Q(A)$ is infinite. On the other hand, since $\Tdeg(B) = \Tdeg(A)$, then $Q(B)$ and $Q(A)$ have the same stratiform length. This implies that $Q(B)$ is finitely generated as a $Q(A)$-module, contradicting $\dim_{Q(B)} Q(A)=\infty$. 
\end{proof}

\begin{theorem}\label{thm:stratiform}
Let $A$ be a noetherian Poisson stratiform domain such that $\DD(1)=A$. Then $A$ is strongly Poisson skew cancellative in the category of noetherian Poisson stratiform domains.
\end{theorem}
\begin{proof}
Let $B$ be another noetherian Poisson stratiform domain. Set $\overline{A} = A[t_1;\alpha_1,\delta_1]_P \cdots [t_n;\alpha_n,\delta_n]_P$ and $\overline{B}=B[t_1';\alpha_1',\delta_1']_P \cdots [t_n';\alpha_n',\delta_n']_P$ and suppose that $\phi:\overline{A} \to \overline{B}$ is a Poisson algebra isomorphism. Note that $\Tdeg(\overline{A})=\Tdeg(A)+n$. By Lemma \ref{lem.dd},
$\phi(A)=\phi(\DD_{\overline{A}}(1))=\DD_{\overline{B}}(1) \subset B$. Set $B'=\phi\inv(B)$ so that $A \subset B'$ and $\Tdeg(Q(B'))=\Tdeg(Q(B))$. Then
\begin{align*}
    \Tdeg(Q(B')) 
        = \Tdeg(Q(B)) = \Tdeg(Q(\overline{B})) - n
        = \Tdeg(Q(\overline{A}))-n = \Tdeg(Q(A)).
\end{align*}
Thus, $\Tdeg(Q(A))=\Tdeg(Q(B'))$ so $A=B'$ by Proposition \ref{prop.strat1}. Hence, $\phi:A \to B$ is a Poisson isomorphism. 
\end{proof}

\section{Remarks and questions}
\label{sec.questions}

We conclude this paper with some remarks and ideas for further study.

Let $A$ and $B$ be Poisson algebras. We refer the reader to \cite{OH3} for background on Poisson modules and the Poisson enveloping algebra. We say $A$ and $B$ are \emph{Poisson Morita equivalent} if there is an equivalence $\PMod A \equiv \PMod B$ between their respective Poisson module categories. There is an equivalence $\PMod A \equiv \Mod U(A)$, where $\Mod U(A)$ denotes the category of left modules over the Poisson enveloping algebra of $A$. Thus, we may equivalently define Poisson Morita equivalence by the condition that $\Mod U(A) \equiv \Mod U(B)$.

In \cite{LWZ2,TZZ}, the authors consider a version of cancellation related to Morita equivalence. In that spirit, we say a Poisson algebra $A$ is \emph{universally Poisson Morita cancellative} if $\PMod(A \tensor R) \equiv \PMod(B \tensor R)$ implies $\PMod(A) \equiv \PMod(B)$ for every Poisson algebra $B$ and every finitely generated commutative domain $R$ over $\kk$ with trivial Poisson bracket such that the natural map $\kk \rightarrow R \rightarrow R/I$ is an isomorphism for some Poisson ideal $I \subset R$.
In the special case that $R=\kk[t]$ (resp. $\kk[t_1,\hdots,t_n]$ for all $n \geq 1$), we say $A$ is \emph{Poisson Morita cancellative} (resp. \emph{strongly Poisson Morita cancellative}). 
Of course, by the above discussion, we can make the replacements $\Mod(U(A) \tensor R)$ and $\Mod(U(B) \tensor R)$ above.
Note that the usual (algebraic) definitions for Morita cancellation can be obtained by deleting Poisson in the above definitions.

\begin{example}\label{ex:morita}
Let $\cP_n$ denote the $n$th Weyl Poisson algebra over a field of characteristic zero. Then $U(\cP_n) \iso A_{2n}$, the $2n$th Weyl algebra \cite{UU}.  Let $B$ be a Poisson algebra and let $R$ be a Poisson algebra with trivial bracket such that $\PMod(\cP_n \tensor R) \equiv \PMod(B \tensor R)$. Now $\Mod(U(\cP_n \tensor R)) \equiv \Mod(U(B \tensor R))$. By \cite[Theorem 5.5]{LWZ6} and \cite[Lemma 2]{UU}, $U(\cP_n \tensor R) \iso A_{2n} \tensor U(R)$ and $U(B \tensor R) \iso U(B) \tensor U(R)$. As $A_{2n}$ is universally Morita cancellative \cite[Theorem 0.4]{TZZ}, then we have $\Mod U(\cP_n) \equiv \Mod U(B)$. That is, $\PMod \cP_n \equiv \PMod B$.
Hence, $\cP_n$ is universally Poisson Morita cancellative.
\end{example}

\begin{question}
Under what conditions is a Poisson algebra $A$ Poisson Morita cancellative?
\end{question}

In particular by applying the same argument in Example \ref{ex:morita} one can show that for any Poisson algebra $A$ if the Poisson enveloping algebra $U(A)$ is strongly Morita cancellative, then $A$ is strongly Poisson Morita cancellative. So we are interested in the other direction. 

\begin{question}
Let $A$ be any Poisson algebra. Does $A$ being (strongly) Poisson Morita cancellative imply that the Poisson enveloping algebra $U(A)$ is (strongly) Morita cancellative?
\end{question}

Moreover in \cite[Definition 0.2]{LWZ2} a notion of {\it derived cancellation} is further introduced for any associative algebra by replacing the respective  module category by its derived category. We are interested in its Poisson version. 

\begin{question}
Can one define a suitable notion of \emph{derived Poisson cancellation}? Under what conditions does a Poisson algebra satisfy this condition?
\end{question}

In Theorem \ref{thm.3var}, we showed that every quadratic Poisson algebra in three variables with nontrivial Poisson bracket is cancellative assuming $\ch(\kk)=0$. 

\begin{question}\label{q:pos}
To what extent is Theorem \ref{thm.3var} true in positive characteristic.
\end{question}

It is important to point out that if a Poisson algebra $A$ has zero Poisson bracket then $A$ is Poisson cancellative if and only if it is cancellative in the original sense of Zariski. One remarkable achievement in Zariski cancellation problem for polynomial algebras is a result of Gupta in 2014 \cite{Gu1, Gu2} which settled the problem negatively in positive characteristic for $\kk[x_1,\ldots,x_n]$ with $n\ge 3$. The Zariski cancellation problem in characteristic zero remains open for $\kk[x_1,\ldots,x_n]$ with $n\ge 3$. So, regarding Question \ref{q:pos} above, we expect there exist nonzero quadratic Poisson brackets on $\kk[x_1,\ldots,x_n]$ for each $n\ge 3$ and $\ch(\kk)>0$ where the corresponding quadratic Poisson algebra is not Poisson cancellative.  

Finally, in Theorem \ref{thm:Lie}, we showed that $PS(\fg)$ is cancellative for all nontrivial Lie algebras $\fg$ of dimension at most $3$. It would be interesting to study cancellation of $PS(\fg)$ for higher-dimensional Lie algebras. More generally, one can ask the following question.

\begin{conjecture}
Let $\fg$ be a finite-dimensional Lie algebra. Then $PS(\fg)$ is Poisson cancellative if and only if $U(\fg)$ is cancellative.
\end{conjecture}

An even more general conjecture regarding the relationship of cancellation properties between a Poisson algebra and its quantization can be stated as follows.

\begin{conjecture}
Suppose $Z$ is a Poisson algebra and $A$ is a quantization of $Z$. Then $Z$ is Poisson cancellative if and only if $A$ is cancellative. 
\end{conjecture}

\bibliographystyle{plain}

\end{document}